\documentclass[12pt]{amsart}
\usepackage{amsmath,amssymb}
\usepackage{amsfonts}
\usepackage{amsthm}
\usepackage{latexsym}
\usepackage{graphicx}

\def\p{\partial}

\def\R{\mathbb{R}}

\def\vv<#1>{\langle#1\rangle}

\def\XXint#1#2{\setbox0=\hbox{$#1{#2}{\int}$}{#2}\kern-.5\wd0 }

\def\XXint#1#2#3{{\setbox0=\hbox{$#1{#2#3}{\int}$}
     \vcenter{\hbox{$#2#3$}}\kern-.5\wd0}}

\def\vv<#1>{{\left\langle#1\right\rangle}}

\def\Vol{\mbox{Vol}}

\newtheorem{thm}{Theorem}[section]

\newtheorem{lem}{Lemma}[section]

\newtheorem{cor}{Corollary}[section]
\theoremstyle{definition}

\theoremstyle{remark}

\newtheorem{rem}{Remark}[section]

\numberwithin{equation}{section}

\begin{document}
\title{Trace and inverse trace of Steklov eigenvalues}

\author{Yongjie Shi}
\address{Department of Mathematics, Shantou University, Shantou, Guangdong, 515063, China}
\email{yjshi@stu.edu.cn}
\author{Chengjie Yu$^1$}
\address{Department of Mathematics, Shantou University, Shantou, Guangdong, 515063, China}
\email{cjyu@stu.edu.cn}
\thanks{$^1$Research partially supported by a supporting project from the Department of Education of Guangdong Province with contract no. Yq2013073, the Yangfan project from Guangdong Province and NSFC 11571215.}
\renewcommand{\subjclassname}{%
  \textup{2010} Mathematics Subject Classification}
\date{}
\begin{abstract}
In this paper, we obtain some new estimates for the trace and inverse trace of  Steklov eigenvalues. The estimates generalize some previous results of Hersch-Payne-Schiffer \cite{HPS}, Brock \cite{Br}, Raulot-Savo \cite{RS1} and Dittmar \cite{Di}.
\end{abstract}
\maketitle\markboth{Shi \& Yu}{Trace and inverse trace}
\section{Introduction}
Let $(M^n,g)$ be a compact oriented Riemannian manifold with nonempty boundary. The Dirichlet-to-Neumann map or Steklov operator for differential forms sends a differential $p$-form $\omega\in A^p(\p M)$  to $i_\nu d\hat\omega$. Here $\nu$ is the outward unit normal vector and $\hat\omega$ is the tangential harmonic extension of $\omega$. That is,
\begin{equation}
\left\{\begin{array}{l}\Delta \hat\omega=0\\
\iota^*\hat\omega=\omega\\
i_\nu\hat\omega=0
\end{array}\right.
\end{equation}
where $\Delta=d\delta+\delta d$ is the Hodge-Laplace operator and $\iota:\p M\to M$ is the natural inclusion. This definition of Dirichlet-to-Neumann map for differential forms was introduced by Raulot and Savo in \cite{RS2} in recent years. When $p=0$, it is clear that the definition is the same as the classical one that essentially introduced by Steklov \cite{St}. Moreover, it was shown in \cite{RS2} that, the same as the classical one (See \cite{Ta}), the Dirichlet-to-Neumann map is a nonnegative self-adjoint first order elliptic pseudo differential operator. So, the eigenvalues of the Dirichlet-to-Neumann map for differential $p$-forms are nonnegative and discrete,  and we can list them in ascending order (counting multiplicity) as
\begin{equation}
0\leq \sigma_1^{(p)}\leq \sigma_2^{(p)}\leq \cdots\leq \sigma_k^{(p)}\leq\cdots.
\end{equation}
They are called Steklov eigenvalues of $(M,g)$ for differential $p$-forms.

There are two other definitions of Dirichlet-to-Neumann map introduced by Joshi-Lionheart \cite{JL} and Belishev-Sharafutdinov \cite{BS} considering geometric inverse problems. However, the definition of Raulot-Savo \cite{RS2} given above is more suitable for spectral analysis because it is self-adjoint and elliptic. The Dirichlet-to-Neumann map for functions has been extensively studied because it is deeply related with physics (See \cite{Ku}) and Calder\'on's inverse problem in applied mathematics (See \cite{Ca,Ul}).

There have been many works on the estimate of Steklov eigenvalues.  For example,
in \cite{W}, Weinstock obtained the following estimate for any simply connected planar domain $\Omega$:
\begin{equation}\label{eqn-W}
\sigma_2^{(0)}L(\p\Omega)\leq 2\pi
\end{equation}
 where equality holds if and only if $\Omega$ is a disk. Here $L(\p\Omega)$ means the length of the $\p \Omega$. This estimate was later generalized by Hersch-Payne-Schiffer \cite{HPS} in a much more general form by using the conjugate harmonic functions. More precisely, Hersch-Payne-Schiffer \cite{HPS} obtained the following inequalities for any simply connected planar domain $\Omega$:
\begin{equation}\label{eqn-HPS}
\sigma_{p+1}^{(0)}\sigma_{q+1}^{(0)}L(\p\Omega)^2\leq \left\{\begin{array}{ll}(p+q)^2\pi^2& p+q\ \mbox{is even}\\
(p+q-1)^2\pi^2& p+q\ \mbox{is odd.}\end{array}\right.
\end{equation}
Letting $p=q$ in \eqref{eqn-HPS}, one obtain
\begin{equation}\label{eqn-g-W}
\sigma_{p+1}^{(0)}L(\p\Omega)\leq 2p\pi.
\end{equation}
This is a generalization of \eqref{eqn-W}. \eqref{eqn-HPS} and \eqref{eqn-g-W} were later generalized by \cite{FS1,GP1,GP2,GP3} for general surfaces. \eqref{eqn-HPS} was recently generalized by \cite{YY2} for general manifolds using the theory of Raulot-Savo \cite{RS1} on Steklov eigenvalues of differential forms and harmonic conjugate forms. There are many other interesting estimates of Steklov eigenvalues, see for example  \cite{Br,CEG,E,FS2,IM,K,RS1,RS2,RS3,WX}. \cite{GP} is an excellent survey for recent progresses on the topic.

In this paper, motivated by Brock \cite{Br} and Hersch-Payne-Schiffer \cite{HPS}, by combining the tricks in \cite{YY,YY2} and \cite{HPS}, we consider lower bound estimate for inverse trace and upper bound estimate for trace of Steklov eigenvalues.

We first consider inverse trace of Steklov eigenvalues for functions on Riemannian surfaces. In this case, we obtain the following result.
\begin{thm}\label{thm-g-steklov}
Let $(M^2,g)$ be a compact oriented Riemannian surface with $\p M\neq \emptyset$. Then
\begin{equation}
\begin{split}
\sum_{i=1}^{2n}f\left(\frac{1}{\sigma_{m+i}^{(0)}}\right)\geq2\sum_{i=1}^{n}f\left(\frac{1}{\lambda_{b_1+2m+2i-2}^{1/2}}\right)
\end{split}
\end{equation}
for any increasing convex function $f$, and any positive integer $m$, where $b_1$ is the first Betti number of $M$ and
\begin{equation}
0=\lambda_1\leq \lambda_2\leq \cdots\leq \lambda_k\leq \cdots
\end{equation}
are the spectrum of $\p M$ for Laplacian operator.
\end{thm}

In Theorem \ref{thm-g-steklov}, when $b_1=0$, if we choose $m=1$ and $f(t)=t$, by noting that  $\lambda_{2i}(\p M)=\left(\frac{2i\pi}{L(\p M)}\right)^2$ with $L(\p M)$ the length of $\p M$ ($\p M$ has only one connected component), we have
\begin{equation}
\frac{1}{\sigma_2^{(0)}}+\frac{1}{\sigma_3^{(0)}}+\cdots+\frac{1}{\sigma_{2n+1}^{(0)}}\geq \frac{L(\p M)}{\pi}\left(1+\frac{1}{2}+\frac{1}{3}+\cdots+\frac{1}{n}\right).
\end{equation}
This is the inverse trace estimate in \cite{HPS}.

Moreover, when $b_1=0$, if we choose $m=1$ and $f(t)=t^2$ in Theorem \ref{thm-g-steklov}, we have
\begin{equation}
\frac{1}{{\sigma_2^{(0)}}^2}+\frac{1}{{\sigma_3^{(0)}}^2}+\cdots+\frac{1}{{\sigma_{2n+1}^{(0)}}^2}\geq \frac{L(\p M)^2}{2\pi^2}\left(1+\frac{1}{2^2}+\frac{1}{3^2}+\cdots+\frac{1}{n^2}\right).
\end{equation}
By letting $n\to\infty$ in the last inequality, we have
\begin{equation}
\sum_{i=2}^\infty\frac{1}{{\sigma_{i}^{(0)}}^2}\geq \frac{L(\p M)^2}{12}.
\end{equation}
This is an estimate stronger than Theorem 3.3 in \cite{Di}.

Next, we consider the inverse trace of Steklov eigenvalues for general manifolds. For this case, we obtain the following result.

\begin{thm}\label{thm-g-steklov-man}
Let $(M^n,g)$ be a compact oriented Riemannian manifold with nonempty boundary. Then
\begin{enumerate}
\item \begin{equation}
\sum_{i=1}^{m}f\left(\frac{1}{\sigma_{r+i}^{(0)}}\right)+\sum_{i=1}^mf\left(\frac{1}{\sigma_{b_{n-2}+s+i-1}^{(n-2)}}\right)\geq 2\sum_{i=1}^mf\left(\frac{1}{\lambda_{b_{n-1}+r+s+i-1}^{1/2}}\right) \end{equation}
\item \begin{equation}
\sum_{i=1}^{m}f\left(\frac{1}{\sigma_{r+i}^{(0)}\sigma_{b_{n-2}+s+i-1}^{(n-2)}}\right)\geq\sum_{i=1}^mf\left(\frac{1}{\lambda_{b_{n-1}+r+s+i-1}}\right)
\end{equation}
\end{enumerate}
for any increasing convex function $f$ and any positive integers $r,s$, where $b_p$ is the $p$-th Betti number of $M$ and
\begin{equation}
0=\lambda_1\leq \lambda_2\leq \cdots\leq \lambda_k\leq \cdots
\end{equation}
are the spectrum of $\p M$ for Laplacian operator.
\end{thm}
When $n=2$, the theorem also gives some estimates for Riemannian surfaces. For example, when $r=s$, we have
\begin{equation}
\sum_{i=1}^{m}f\left(\frac{1}{\sigma_{r+i}^{(0)}}\right)\geq \sum_{i=1}^mf\left(\frac{1}{\lambda_{b_{1}+2r+i-1}^{1/2}}\right).
\end{equation}
This is weaker than Theorem \ref{thm-g-steklov}. However, it is also sharp when $M$ is a disk.  Moreover, when $m=1$ and $f(t)=t$ in Theorem \ref{thm-g-steklov-man}, we obtain
\begin{equation}
\sigma_{1+r}^{(0)}\sigma_{b_{n-2}+s}^{(n-2)}\leq \lambda_{b_{n-1}+r+s}.
\end{equation}
This is the higher dimensional generalization of \eqref{eqn-HPS} in \cite{YY2}.

Finally, we obtain an estimate for trace of Steklov eigenvalues.
\begin{thm}\label{thm-trace-g}
Let $(M^n,g)$ be a compact oriented Riemannian manifold with nonempty boundary. Let $V$ be the space of parallel exact 1-forms on $M$. Suppose that $\dim V=m>0$. Then
\begin{equation}
\sum_{i=1}^{C_{m}^{p+1}}\sigma_{b_{p}+i}^{(p)}\leq \frac{C_{m-1}^p\Vol(\p M)}{\Vol (M)}
\end{equation}
for $p=1,2,\cdots,m-1$.
\end{thm}

This result implies the general estimate in \cite{YY} directly (See Corollary \ref{cor-YY-1} ). Applying the result to domains in Euclidean spaces, one can obtain the estimates in \cite{Br, YY} and \cite{RS1} (See Corollary \ref{cor-YY-2} and Remark \ref{rem-br}).

The proof of Theorem \ref{thm-g-steklov} and Theorem \ref{thm-g-steklov-man} is motivated by \cite{HPS} using conjugate harmonic forms. The proof of Theorem \ref{thm-trace-g} is motivated by \cite{Br}. The organization of the remaining parts of this paper is as follows. In Section 2, we recall some preliminaries in conjugate harmonic forms and matrix inequalities. In Section 3, we prove Theorem \ref{thm-g-steklov}. In Section 4, we prove Theorem \ref{thm-g-steklov-man}. Finally, in Section 5, we prove Theorem \ref{thm-trace-g}.
\section{Preliminaries}
In this section, we recall some preliminaries in conjugate harmonic forms and matrix inequalities that will be used in next sections.

First, we summarize the construction of conjugate harmonic forms of a harmonic function on a general manifold in \cite{YY2} as the following lemma. The proof which can be found in \cite{YY2} is a simple application of Hodge decomposition theorem for compact Riemannian manifolds with nonempty boundary (See \cite{S}).
\begin{lem}\label{lem-harm-conj}
Let $(M^n,g)$ be a compact oriented Riemannian manifold with nonempty boundary and $u$ be a harmonic function on $M$.  Suppose that
\begin{equation}
*du\perp_{L^2(M)}\mathcal H_N^{(n-1)}(M).
\end{equation}
Then, there is a unique $\omega\in A^{n-2}(M)$ such that
\begin{enumerate}
\item $d\omega=*du;$
\item $\delta\omega=0$;
\item $i_\nu\omega=0$ and
\item $\omega\perp_{L^2(\p M)}\mathcal H_N^{n-2}(M)$.
\end{enumerate}
Here
\begin{equation}
\mathcal H_N^{p}=\{\gamma\in A^p(M)\ |\ d\gamma=\delta\gamma=0\ \mbox{and}\ i_\nu\gamma=0\}.
\end{equation}
$\omega$ is called the conjugate harmonic form of $u$.
\end{lem}

Next, recall an inequality on the trace of eigenvalues for positive definite matrices that will be used in next sections. Because we can not find direct reference for the inequality, we also give the proof of the inequality.
\begin{lem}\label{lem-matrix}
Let $A$ and $B$ be two symmetric $n\times n$ matrices that are both positive definite. Let
$$\lambda_1(A)\leq \lambda_2(A)\leq \cdots\leq\lambda_n(A)$$
and
$$\lambda_1(B)\leq \lambda_2(B)\leq \cdots\leq\lambda_n(B)$$
eigenvalues of $A$ and $B$ respectively.
Then, for any increasing convex function $f$,
\begin{equation}
\sum_{i=1}^nf\left(\lambda_i(A)\lambda_i(B)\right)\geq \sum_{i=1}^nf(A(i,i)B(i,i))
\end{equation}
where $A(i,j)$ and $B(i,j)$ are the $(i,j)$-entries of $A$ and $B$ respectively.
In particular, by letting $B=I_n$, we have
\begin{equation}
\sum_{i=1}^nf\left(\lambda_i(A)\right)\geq \sum_{i=1}^nf(A(i,i)).
\end{equation}
\end{lem}
\begin{proof}
Let $A\circ B$ be the Hadamard product of $A$ and $B$. That is,
$$A\circ B(i,j)=A(i,j)B(i,j)$$
for $i,j=1,2,\cdots,n$.
Then, by basic majorization relations (See \cite{HJ} or \cite[Theorem 2.6]{Z}),
\begin{equation}\label{eqn-major-1}
\begin{split}
&\{\lambda_1(A\circ B),\lambda_2(A\circ B),\cdots,\lambda_n(A\circ B)\}\\
\prec_w&\{\lambda_1(A)\lambda_1(B),\lambda_2(A)\lambda_2(B),\cdots,\lambda_n(A)\lambda_n(B)\}.
\end{split}
\end{equation}
Moreover, by Schur's Theorem (See  \cite[Theorem 2.1]{Z}),
\begin{equation}\label{eqn-major-2}
\begin{split}
&\{A(1,1)B(1,1),A(2,2)B(2,2),\cdots,A(n,n)B(n,n)\}\\
\prec&\{\lambda_1(A\circ B),\lambda_2(A\circ B),\cdots,\lambda_n(A\circ B)\}.
\end{split}
\end{equation}
Combining \eqref{eqn-major-1} and \eqref{eqn-major-2}, we have
\begin{equation}
\begin{split}
&\{A(1,1)B(1,1),A(2,2)B(2,2),\cdots,A(n,n)B(n,n)\}\\
\prec_w&\{\lambda_1(A)\lambda_1(B),\lambda_2(A)\lambda_2(B),\cdots,\lambda_n(A)\lambda_n(B)\}.
\end{split}
\end{equation}
Then, by majorization principles (See \cite{MO} or \cite[Theorem 2.3]{Z}),
\begin{equation}
\begin{split}
&\{f(A(1,1)B(1,1)),f(A(2,2)B(2,2)),\cdots,f(A(n,n)B(n,n))\}\\
\prec_w&\{f(\lambda_1(A)\lambda_1(B)),f(\lambda_2(A)\lambda_2(B)),\cdots,f(\lambda_n(A)\lambda_n(B))\}
\end{split}
\end{equation}
for any increasing convex function $f$. This give us the conclusion.
\end{proof}
\section{Proof of Theorem \ref{thm-g-steklov}}
In this section, by using Lemma \ref{lem-harm-conj} and the trick in \cite{HPS}, we prove Theorem \ref{thm-g-steklov}.

\begin{proof}[Proof of Theorem \ref{thm-g-steklov}]Let
$$\phi_1\equiv 1, \phi_2,\cdots,\phi_k,\cdots$$
be a complete orthonormal system of eigenfunctions on $\p M$ corresponding to the spectrum of
$$0=\lambda_1\leq \lambda_2\leq\cdots\leq \lambda_k\leq \cdots$$
of $\p M$. Moreover, let
$$\psi_1,\psi_2,\cdots,\psi_k,\cdots$$
be a complete orthonormal system of eigenfunctions for positive Steklov eigenvalues of functions corresponding to eigenvalues listed in ascending order.

Let $u_1\neq 0$ be the harmonic extension of $c_2\phi_2+\cdots+c_{b_1+2m}\phi_{b_1+2m}$ such that $u_1\perp_{L^2(M)} \mathcal H_N^{1}(M)$ and $u_1, u_2\perp_{L^2(\p M)}\psi_1,\psi_2,\cdots,\psi_{m-1}$. Here $u_2$ is the conjugate harmonic function of $u_1$ as in Lemma \ref{lem-harm-conj}. This can always be done since the restrictions are just $b_1+2m-2$ homogeneous linear equations for $b_1+2m-1$ unknowns $c_2,c_3,\cdots,c_{b_1+2m}$.

Similarly, let $u_3\neq 0$ be the harmonic extension of $c_2\phi_2+\cdots+c_{b_1+2m+2}\phi_{b_1+2m+2}$ such that
\begin{enumerate}
\item $*du_3\perp_{L^2(M)}\mathcal H^1_N(M)$;
\item $u_3,u_4\perp_{L^2(\p M)}\psi_1,\psi_2,\cdots,\psi_{m-1}$ and
\item $d u_3\perp_{L^2(M)}d u_1,d u_2.$
\end{enumerate}
Here $u_4$ is the conjugate harmonic function of $u_3$.

Continuing this process on, we can construct nonconstant harmonic functions $u_1,u_2, u_3,u_4,\cdots,u_{2n-1},u_{2n}$, such that
\begin{enumerate}
\item $u_{2i-1}\in \mbox{span}\{\hat\phi_2,\hat \phi_3,\cdots,\hat\phi_{b_1+2m+2i-2}\}$;
\item $u_{2i}$ is the conjugate harmonic function of $u_{2i-1}$;
\item $u_1,u_2,\cdots,u_{2n}\perp_{L^2{(\p M)}} \psi_1,\psi_2,\cdots,\psi_{m-1}$;
\item $du_{2i-1}\perp_{L^2(M)}du_1,du_2,\cdots, du_{2i-2}$
\end{enumerate}
for $i=1,2,\cdots,n$.

Now, we check that
\begin{equation}
d u_{2i}\perp_{L^2(M)} du_1,du_2,\cdots, du_{2i-1}.
\end{equation}
First,
\begin{equation}
\int_M\vv<du_{2i}, du_{2i-1}>dV_M=\int_M\vv<*du_{2i-1},du_{2i-1}>dV_M=-\int_Mdu_{2i-1}\wedge du_{2i-1}dV_M=0.
\end{equation}
Moreover,
\begin{equation}
\begin{split}
\int_{M}\vv<du_{2i},du_{2j}>dV_M=&\int_M\vv<*du_{2i-1},*du_{2j-1}>dV_M\\
=&\int_M\vv<du_{2i-1},du_{2j-1}>dV_M=0.
\end{split}
\end{equation}
and
\begin{equation}
\begin{split}
\int_M\vv<d u_{2i},du_{2j-1}>dV_M=&\int_M\vv<*du_{2i-1},du_{2j-1}>dV_M\\
=&-\int_M\vv<du_{2i-1},*du_{2j-1}>dV_M=-\int_M\vv<du_{2i-1},du_{2j}>dV_M=0
\end{split}
\end{equation}
for any $j=1,2,\cdots,i-1$.

 Hence, we have shown that
 \begin{equation}
 \int_M\vv<du_i,du_j>dV_M=0
 \end{equation}
 for any $i\neq j$.
 Let $V=\mbox{span}\{u_1,u_2,\cdots,u_{2n-1},u_{2n}\}$ and $A$ be a linear transformation on $V$ such that
\begin{equation}\label{eqn-A}
\int_M\vv<du,d v>=\int_{\p M}\vv<Au,v>dV_{\p M}
\end{equation}
for any $u,v\in V$. Let
$$\lambda_1(A)\leq \lambda_2(A)\leq\cdots\leq  \lambda_{2n}(A)$$
be the eigenvalues of $A$. Then, by the Courant-Fischer's min-max principle,
\begin{equation}\label{eqn-s-l}
\sigma_{m+i}^{(0)}\leq \lambda_i(A)
\end{equation}
for $i=1,2,\cdots,2n$.

Let $v_i=\frac{u_i}{\left(\int_M\vv<du_i,du_i>dV_M\right)^{1/2}}$. We still denote the matrix of $A$ under the basis of $\{v_1,v_2,\cdots,v_{2n}\}$ as $A$. Then, by \eqref{eqn-A}
\begin{equation}
A^{-1}(i,j)=\frac{\int_{\p M}\vv<u_i,u_j>dV_{\p M}}{\left(\int_M\vv<du_i,du_i>dV_M\int_M\vv<du_j,du_j>dV_M\right)^{1/2}}.
\end{equation}
Moreover, by Cauchy-Schwarz inequality,
\begin{equation}\label{eqn-AA}
\begin{split}
&A^{-1}(2i-1,2i-1)A^{-1}(2i,2i)\\
=&\frac{\int_{\p M}\vv<u_{2i-1},u_{2i-1}>dV_{\p M}\int_{\p M}\vv<u_{2i},u_{2i}>dV_{\p M}}{\int_M\vv<du_{2i-1},du_{2i-1}>dV_M\int_M\vv<du_{2i},du_{2i}>dV_M}\\
=&\frac{\int_{\p M}\vv<u_{2i-1},u_{2i-1}>dV_{\p M}\int_{\p M}\vv<u_{2i},u_{2i}>dV_{\p M}}{\int_M\vv<*du_{2i},*du_{2i}>dV_M\int_M\vv<du_{2i},du_{2i}>dV_M}\\
=&\frac{\int_{\p M}\vv<u_{2i-1},u_{2i-1}>dV_{\p M}\int_{\p M}\vv<u_{2i},u_{2i}>dV_{\p M}}{(\int_M\vv<du_{2i},du_{2i}>dV_M)^2}\\
=&\frac{\int_{\p M}\vv<u_{2i-1},u_{2i-1}>dV_{\p M}\int_{\p M}\vv<u_{2i},u_{2i}>dV_{\p M}}{(\int_{\p M}\vv<u_{2i},i_\nu du_{2i}>dV_{\p M})^2}\\
=&\frac{\int_{\p M}\vv<u_{2i-1},u_{2i-1}>dV_{\p M}\int_{\p M}\vv<u_{2i},u_{2i}>dV_{\p M}}{(\int_{\p M}\vv<u_{2i},i_\nu *du_{2i-1}>dV_{\p M})^2}\\
\geq&\frac{\int_{\p M}\vv<u_{2i-1},u_{2i-1}>dV_{\p M}}{\int_{\p M}\vv<du_{2i-1},du_{2i-1}>dV_{\p M}}\\
\geq& \frac 1{\lambda_{b_1+2m+2i-2}}.
\end{split}
\end{equation}
for $i=1,2,\cdots,n$.

Finally, by \eqref{eqn-s-l}, \eqref{eqn-AA}, Lemma \ref{lem-matrix}, and that $f$ is increasing and convex,
\begin{equation}
\begin{split}
\sum_{i=1}^{2n}f\left(\frac{1}{\sigma_{m+i}^{(0)}}\right)\geq& \sum_{i=1}^{2n}f\left(\frac{1}{\lambda_i(A)}\right)\\
\geq&\sum_{i=1}^{n}f\left(A^{-1}(2i-1,2i-1)\right)+f\left(A^{-1}(2i,2i)\right)\\
\geq &2\sum_{i=1}^nf\left(\frac{A^{-1}(2i-1,2i-1)+A^{-1}(2i,2i)}{2}\right)\\
\geq &2\sum_{i=1}^nf\left(\left(A^{-1}(2i-1,2i-1)A^{-1}(2i,2i)\right)^{1/2}\right)\\
\geq&2\sum_{i=1}^nf\left(\frac{1}{\lambda_{b_1+2m+2i-2}^{1/2}}\right)
\end{split}
\end{equation}
This completes the proof of the theorem.
\end{proof}
\section{Proof of Theorem \ref{thm-g-steklov-man}}
In this section, by a similar argument as in the proof of Theorem \ref{thm-g-steklov} using conjugate harmonic forms, we prove Theorem \ref{thm-g-steklov-man}.
\begin{proof}[Proof of Theorem \ref{thm-g-steklov-man}]
Let
$$\phi_1\equiv 1, \phi_2,\cdots,\phi_k,\cdots$$
and
$$\psi_1,\psi_2,\cdots,\psi_k,\cdots$$
be the same as in the proof of Theorem \ref{thm-g-steklov}. Moreover, let
$$\epsilon_1,\epsilon_2,\cdots,\epsilon_k,\cdots$$
be a complete orthonormal system of eigenforms for positive Steklov eigenvalues of differential $(n-2)$-forms corresponding to eigenvalues listed in ascending order.

Similarly as in the proof of Theorem \ref{thm-g-steklov}. We can find nonconstant harmonic functions
$u_1,u_2,\cdots,u_m$ such that
\begin{enumerate}
\item $*du_i\perp_{L^2(M)}\mathcal H^{n-1}_N(M)$;
\item $u_i\perp_{L^2(\p M)}\psi_1,\psi_2,\cdots,\psi_{r-1}$;
\item $\omega_i\perp_{L^2(\p M)}\epsilon_1,\epsilon_2,\cdots,\epsilon_{s-1}$ where $\omega_i$ is the conjugate harmonic form of $u_i$ as in Lemma \ref{lem-harm-conj};
\item $u_i\in \mbox{span}\{\hat \phi_2,\hat\phi_3,\cdots, \hat \phi_{b_{n-1}+r+s+i-1}\}$;
\item $\int_M\vv<du_i,du_j>dV_M=\delta_{ij}$
\end{enumerate}
for $i,j=1,2,\cdots,m$. Moreover, note that
\begin{equation}
\int_{M}\vv<d\omega_i,d\omega_j>dV_M=\int_{M}\vv<du_i,du_j>dV_M=\delta_{ij}
\end{equation}
for $i,j=1,2,\cdots,m$.

Let $V=\mbox{span}\{u_1,u_2,\cdots,u_m\}$ and $W=\mbox{span}\{\omega_1,\omega_2,\cdots,\omega_m\}$. Let $A:V\to V$ and $B:W\to W$
be linear transformation on $V$ and $W$ such that
\begin{equation}\label{eqn-def-A}
\int_M\vv<du,dv>dV_M=\int_{\p M}\vv<Au,v>dV_{\p M}
\end{equation}
for any $u,v\in V$ and
\begin{equation}\label{eqn-def-B}
\int_M\vv<d\alpha,d\beta>=\int_{\p M}\vv<B\alpha,\beta>dV_{\p M}
\end{equation}
for any $\alpha,\beta\in W$ respectively. Then, by Courant-Fischer's min-max principle,
\begin{equation}\label{eqn-s-A}
\sigma_{r+i}^{(0)}\leq \lambda_i(A)
\end{equation}
and
\begin{equation}\label{eqn-s-B}
\sigma_{b_{n-2}+s+i-1}^{(n-2)}\leq \lambda_i(B)
\end{equation}
for $i=1,2,\cdots,m$.

Denote the matrix of $A$ and $B$ under the basis $\{u_1,u_2,\cdots,u_m\}$ and $\{\omega_1,\omega_2,\cdots,\omega_m\}$ as $A$ and $B$ respectively. Then, by \eqref{eqn-def-A} and \eqref{eqn-def-B},
\begin{equation}
A^{-1}(i,j)=\int_{\p M}\vv<u_i,u_j>dV_{\p M}
\end{equation}
and
\begin{equation}
B^{-1}(i,j)=\int_{\p M}\vv<\omega_i,\omega_j>dV_{\p M}
\end{equation}
for $i,j=1,2,\cdots,m$. Moreover, by Cauchy-Schwarz inequality,
\begin{equation}\label{eqn-AB}
\begin{split}
&A^{-1}(i,i)B^{-1}(i,i)\\
=&\frac{\int_{\p M}\vv<u_{i},u_{i}>dV_{\p M}\int_{\p M}\vv<\omega_{i},\omega_{i}>dV_{\p M}}{\int_M\vv<du_{i},du_{i}>dV_M\int_M\vv<d\omega_{i},d\omega_{i}>dV_M}\\
=&\frac{\int_{\p M}\vv<u_{i},u_{i}>dV_{\p M}\int_{\p M}\vv<\omega_{i},\omega_{i}>dV_{\p M}}{\int_M\vv<*d\omega_{i},*d\omega_{i}>dV_M\int_M\vv<d\omega_{i},d\omega_{i}>dV_M}\\
=&\frac{\int_{\p M}\vv<u_{i},u_{i}>dV_{\p M}\int_{\p M}\vv<\omega_{i},\omega_{i}>dV_{\p M}}{(\int_M\vv<d\omega_{i},d\omega_{i}>dV_M)^2}\\
=&\frac{\int_{\p M}\vv<u_{i},u_{i}>dV_{\p M}\int_{\p M}\vv<\omega_{i},\omega_{i}>dV_{\p M}}{(\int_{\p M}\vv<\omega_{i},i_\nu d\omega_{i}>dV_{\p M})^2}\\
=&\frac{\int_{\p M}\vv<u_{i},u_{i}>dV_{\p M}\int_{\p M}\vv<\omega_{i},\omega_{i}>dV_{\p M}}{(\int_{\p M}\vv<\omega_{i},i_\nu *du_i>dV_{\p M})^2}\\
\geq&\frac{\int_{\p M}\vv<u_i,u_i>dV_{\p M}}{\int_{\p M}\vv<du_{i},du_{i}>dV_{\p M}}\\
\geq& \frac 1{\lambda_{b_{n-1}+r+s+i-1}}.
\end{split}
\end{equation}
We are ready to prove the inequalities.
\begin{enumerate}
\item By \eqref{eqn-s-A}, \eqref{eqn-s-B}, \eqref{eqn-AB}, Lemma \ref{lem-matrix}, and that $f$ is increasing and convex,
\begin{equation}
\begin{split}
  \sum_{i=1}^{m}f\left(\frac{1}{\sigma_{r+i}^{(0)}}\right)+\sum_{i=1}^mf\left(\frac{1}{\sigma_{b_{n-2}+s+i-1}^{(n-2)}}\right)\geq& \sum_{i=1}^mf\left(\frac{1}{\lambda_i(A)}\right)+\sum_{i=1}^mf\left(\frac{1}{\lambda_i(B)}\right)\\
  \geq&\sum_{i=1}^mf\left(A^{-1}(i,i)\right)+f\left(B^{-1}(i,i)\right)\\
\geq&2\sum_{i=1}^mf\left((A^{-1}(i,i)B^{-1}(i,i))^{1/2}\right)\\
\geq&2\sum_{i=1}^mf\left(\frac{1}{\lambda_{b_{n-1}+r+s+i-1}^{1/2}}\right).
\end{split}
\end{equation}

(2) By \eqref{eqn-s-A},\eqref{eqn-s-B} , \eqref{eqn-AB} and Lemma \ref{lem-matrix}, and that $f$ is increasing and convex,
\begin{equation}
\begin{split}
\sum_{i=1}^{m}f\left(\frac{1}{\sigma_{r+i}^{(0)}\sigma_{b_{n-2}+s+i-1}^{(n-2)}}\right)\geq&\sum_{i=1}^{m}f\left(\frac{1}{\lambda_i(A)\lambda_i(B)}\right)\\
\geq&\sum_{i=1}^{m}f\left(A^{-1}(i,i)B^{-1}(i,i)\right)\\
\geq&\sum_{i=1}^mf\left(\frac{1}{\lambda_{b_{n-1}+r+s+i-1}}\right).
\end{split}
\end{equation}
\end{enumerate}
\end{proof}
\section{Proof of Theorem \ref{thm-trace-g}}
In this section, we give the proof of Theorem \ref{thm-trace-g} and some simple corollaries of it.
\begin{proof}[Proof of Theorem \ref{thm-trace-g}]
Let $A^{(p+1)}$ be a linear transformation on $\wedge^{p+1}V$ such that
\begin{equation}\label{eqn-A-p}
\int_{M}\vv<A^{(p+1)}\xi,\eta>dV_M=\int_{\p M}\vv<i_\nu \xi,i_\nu\eta>dV_{\p M}
\end{equation}
for any $\xi,\eta\in \wedge^{p+1}V$. Then, by Lemma 2.1 in \cite{YY},
\begin{equation}
\sigma_{b_p+i}^{(p)}\leq \lambda_{i}(A^{(p+1)})
\end{equation}
for $i=1,2,\cdots,C_m^{p+1}$. So,
\begin{equation}
\sum_{i=1}^{C_m^{p+1}}\sigma_{b_p+i}^{(p)}\leq \mbox{tr}(A^{(p+1)}).
\end{equation}
Let $\xi_1,\xi_2,\cdots,\xi_{C_m^{p+1}}$ be an orthonormal basis of $\wedge^{(p+1)}V$. That is
\begin{equation}
\vv<\xi_i,\xi_j>=\delta_{ij}
\end{equation}
for $i,j=1,2,\cdots,C_m^{p+1}$. By \eqref{eqn-A-p}, we know that
\begin{equation}\label{eqn-tr-A}
\begin{split}
\mbox{tr}A^{(p+1)}=\frac{1}{\Vol(M)}\sum_{i=1}^{C_{m}^{p+1}}\int_{\p M}\|i_\nu \xi_i\|^2dV_{\p M}=\frac{1}{\Vol(M)}\int_{\p M}\sum_{i=1}^{C_{m}^{p+1}}\|i_{\nu^{\top}} \xi_i\|^2dV_{\p M}\\
\end{split}
\end{equation}
where $\nu^\top$ is the orthogonal projection of $\nu$ onto the dual of $V$.

Let $e_1,e_2,\cdots,e_m$ be an orthogonal basis of the dual of $V$ and $\omega_1,\omega_2,\cdots,\omega_m$ be their dual. Without loss of generality, we can assume that $v^{\top}=\nu_1e_1$ with $0<\nu_1\leq 1$ and $\{\xi_1,\xi_2,\cdots,\xi_m\}$ are just $\{\omega_{i_1}\wedge\omega_{i_2}\wedge\cdots\wedge\omega_{i_{p+1}}\ |\ 1\leq i_1<i_2<\cdots<i_{p+1}\leq m\}$. Then
\begin{equation}\label{eqn-norm}
\sum_{i=1}^{C_{m}^{p+1}}\|i_{\nu^{\top}} \xi_i\|^2=\nu_1^2\sum_{2\leq i_2<i_3<\cdots<i_{p+1}\leq m}\|\omega_{i_2}\wedge\cdots\wedge\omega_{i_{p+1}}\|^2\leq C_{m-1}^{p}.
\end{equation}
Combining \eqref{eqn-tr-A} and \eqref{eqn-norm}, we get the conclusion.
\end{proof}

As a corollary, we get the following estimate in \cite{YY}.
\begin{cor}\label{cor-YY-1}
Let $(M^n,g)$ be a compact oriented Riemannian manifold with nonempty boundary. Let $V$ be the space of parallel exact 1-forms on $M$. Suppose that $\dim V=m>0$. Then
\begin{equation}
\sigma_{b_{p}+i}^{(p)}\leq \frac{C_{m-1}^p}{C_m^{p+1}+1-i}\frac{\Vol(\p M)}{\Vol (M)}
\end{equation}
for $p=1,2,\cdots,m-1$ and $i=1,2,\cdots,C_m^{p+1}$
\end{cor}
\begin{proof} By Theorem \ref{thm-trace-g},
\begin{equation}
(C_m^{p+1}+1-i)\sigma_{b_p+i}^{(p)}\leq \sum_{k=i}^{C_{m}^{p+1}}\sigma_{b_p+k}^{(p)}\leq \frac{C_{m-1}^p\Vol(\p M)}{\Vol (M)}.
\end{equation}
This gives us the conclusion.
\end{proof}
Applying Theorem \ref{thm-trace-g} to the Euclidean case, we have the following inequality.
\begin{cor}\label{cor-YY-2}
Let $\Omega$ be a bounded domain with smooth boundary in $\R^n$. Then
\begin{equation}\label{eqn-tr-euclidean}
\sum_{i=1}^{C_n^{p+1}}\sigma_{b_{p}+i}^{(p)}\leq \frac{C_{n-1}^{p}\Vol(\p\Omega)}{\Vol(\Omega)}.
\end{equation}
for $p=0,1,2,\cdots,n-1$.
\end{cor}
\begin{rem}\label{rem-br}
By Cauchy-Schwarz inequality and \eqref{eqn-tr-euclidean}, it is not hard to see that
\begin{equation}
\sum_{i=1}^{C_{n}^{p+1}}\frac{1}{\sigma_{b_{p}+i}^{(p)}}\geq \frac{nC_{n}^{p+1}\Vol(\Omega)}{(p+1)\Vol(\p\Omega)}.
\end{equation}
This is a generalization of Theorem 1 in \cite{Br}.
\end{rem}

\end{document}